\documentclass[11pt,letterpaper]{article}

\usepackage[margin=1in]{geometry}
\usepackage{amsmath,amssymb,amsthm,bm}
\usepackage{graphicx}
\graphicspath{{figures/}}
\usepackage{booktabs}
\usepackage{comment}
\usepackage{placeins}
\usepackage[numbers,sort&compress]{natbib}
\usepackage[hidelinks]{hyperref}
\hypersetup{
  pdftitle={Exact Series and Error-Controlled Evaluation of the Nonconvex ell-p Proximity Operator},
  pdfauthor={Lixin Shen and Jiangyu Yu},
  pdfsubject={Nonconvex proximity operators and error-controlled evaluation},
  pdfkeywords={proximity operator, nonconvex optimization, Lagrange--Buermann inversion, Mellin--Barnes integral, inexact proximal gradient}
}

\DeclareMathOperator{\Log}{Log}

\DeclareMathOperator{\prox}{prox}
\newcommand{\eps}{\varepsilon}

\theoremstyle{plain}
\newtheorem{theorem}{Theorem}[section]
\newtheorem{lemma}[theorem]{Lemma}
\newtheorem{proposition}[theorem]{Proposition}
\newtheorem{corollary}[theorem]{Corollary}
\theoremstyle{definition}
\newtheorem{algorithm}[theorem]{Algorithm}

\title{Exact Series and Error-Controlled Evaluation of the Nonconvex $\ell_p$ Proximity Operator}
\author{Lixin Shen\thanks{Corresponding author.} \quad Jiangyu Yu\\[0.5ex]
\small Department of Mathematics, Syracuse University, Syracuse, NY 13244, USA\\
\small \texttt{lshen03@syr.edu} \\
\small Department of Mathematics and Statistics, Aubrun University, Auburn, AL 36849\\
\small \texttt{jiy0004@auburn.edu}}
\date{}

\begin{document}

\maketitle

\begin{abstract}
We study the scalar proximity operator of the nonconvex power penalty with exponent between zero and one. Lagrange--B\"urmann inversion gives an exact series for the larger stationary root, and the series converges uniformly throughout the active domain of the proximity operator. For rational exponents, an explicit residue-class decomposition reorganizes the coefficients into finite sums of generalized hypergeometric functions, recovering the familiar special cases and producing an additional formula for \(p=\tfrac13\). A Mellin--Barnes formula yields an equivalent residue-corrected representation and analytic truncation bounds. We distinguish the proximal activation threshold from the lower saddle-node threshold and prove that the active proximal branch is uniformly well-conditioned; the practical difficulty near activation is slow series convergence when the exponent approaches one, rather than singular root sensitivity. We also formulate an inexact proximal-gradient step using a computable optimality residual and a descent test. With summable descent errors and vanishing residuals, the objective values converge, the steps vanish, and every cluster point is critical. Tests over a broad parameter grid and in sparse-recovery proximal-gradient runs illustrate both evaluator accuracy and the implementability of the acceptance certificates.
\end{abstract}



\section{Introduction}

Sparsity priors are central to high-dimensional learning and inverse problems. The $\ell_0$ penalty models sparsity directly but leads to combinatorial optimization, whereas the convex $\ell_1$ relaxation is computationally attractive but can over-shrink large coefficients \cite{Candes-Romberg-Tao:IEEE-TIT:06,Candes-Tao:IEEE-TIT:06}. The nonconvex $\ell_p$ penalty with $0<p<1$ lies between these choices: it promotes sparsity more strongly than $\ell_1$ while reducing bias, but it also complicates analysis and computation \cite{Candes-Wakin-Boyd:JFAA:08,chartrand2008rip,foucart2009lq,LaiWang2021,sun2012lq}. Moreover, empirical studies indicate that natural-image wavelet coefficients are often well modeled by \(p\)-values in the range \([0.5,0.8]\) \cite{Simoncelli1999Bayesian}.

A core primitive in first-order methods for nonsmooth regularization is the proximity operator (see, e.g., \cite{AttouchBolteSvaiter2013KL,Beck-Teboulle:SIAMIS:09,Combettes-Wajs:MMS:05,li2015multi,Micchelli-Shen-Xu:IP-11}). Recent work has also emphasized explicit and structured computation of proximity operators for broad classes of penalties \cite{BricenoAriasVivarVargas2024,JiaPraterBennetteShen2026}. For a function $f$, the operator is defined by
\begin{equation*}
\prox_{\lambda f}(y)\in\arg\min_x \Big\{\tfrac12\|{x-y}\|_2^2+\lambda f(x)\Big\},
\end{equation*}
and it reduces componentwise to a scalar problem when $f(x)=\|x\|_p^p$. Hard- and soft-thresholding give closed forms for $\ell_0$ and $\ell_1$, and formulas are also known for $p=\tfrac12$ and $p=\tfrac23$ \cite{Cao-Sun-Xu:JVCIR:2013,chen2016computing}. For a general exponent, Newton and bisection are available \cite{chen2016computing,liu2024bisection}; our purpose is to provide a uniform analytic representation and an alternative evaluation mechanism with transparent truncation error.

The scalar $\ell_p$ proximity operator is governed by the input activation threshold $\tau_{p,\lambda}$ and the corresponding nonzero output level $\rho_{p,\lambda}$, both defined explicitly in Section~\ref{sec:preliminary}. For $|y|>\tau_{p,\lambda}$, its value is $\operatorname{sgn}(y)x_+(|y|)$, where $x_+(|y|)$ is the larger positive root of
\begin{equation*}
x - |y| + \lambda p\,x^{p-1}=0,\qquad x>0.
\end{equation*}
At $|y|=\tau_{p,\lambda}$, both zero and $\operatorname{sgn}(y)\rho_{p,\lambda}$ are minimizers. Thus, evaluating the nonzero output amounts to following one specified inverse branch rather than finding an arbitrary stationary point.

When \(p=r/s\in(0,1)\) is rational and written in lowest terms, the positive stationarity equation can be transformed into a trinomial algebraic equation. Classical and recent power-series and hypergeometric representations of related trinomial roots \cite{glasser2000hypergeometric,Mane2026Trinomial}, together with Mellin--Barnes representations of trinomial algebraic functions \cite{KocarTanabe2024Trinomial}, therefore provide relevant background. Our contribution is not the mere existence of such representations in the rational case. Rather, we identify the real branch selected by the global proximal minimizer, derive a unified inverse-series and Mellin--Barnes framework valid for every real \(p\in(0,1)\), including irrational \(p\), prove uniform convergence on the full active proximal domain, and connect the evaluation error to computable conditions for an inexact proximal-gradient method.

We recast the normalized stationarity equation as an analytic inverse problem near the regular point $v=1$. The main contributions are as follows.

\begin{itemize}
  \item Lagrange--B\"urmann inversion yields an exact power series for $x_+(y)$ for every $0<p<1$. The series converges uniformly on the entire active domain of the proximity operator.
  \item For every rational exponent, we give an explicit residue-class formula that reduces the inverse series to a finite sum of generalized hypergeometric functions. This recovers the classical cases and gives a compact representation for $p=\tfrac13$.
  \item An equivalent Mellin--Barnes representation leads to a residue-corrected finite-contour evaluator. We separate analytic tail error from numerical quadrature error, so the strength of the output guarantee is stated explicitly.
  \item We distinguish the proximal activation threshold from the lower saddle-node threshold. The active proximal branch satisfies a uniform condition-number bound, while the worst active-domain series convergence factor approaches one only as $p\uparrow1$.
  \item For proximal-gradient algorithms, we replace a distance-only inexactness condition by a computable optimality residual and a descent test. With summable descent errors and vanishing residuals, objective values converge, steps vanish, and every cluster point is critical.
\end{itemize}

Section~\ref{sec:preliminary} reviews the threshold structure. Section~\ref{sec:series} derives the inverse series, Section~\ref{sec:special} gives rational-exponent representations, and Section~\ref{sec:MB} develops the Mellin--Barnes formula, the evaluator, and its use in proximal gradient. Section~\ref{sec:conclusion} concludes. A separate supplement gives a Galois-theoretic classification of radical formulas.

\section{Preliminaries and Threshold Structure}\label{sec:preliminary}
In this section, we recall the threshold structure of the $\ell_p$ proximity operator for $0<p<1$ and identify the larger stationary branch selected by the global minimizer. Define \(\|x\|_p^p=\sum_{i=1}^n |x_i|^p\). We study the proximity operator
\[
\prox_{\lambda\|\cdot\|_p^p}(y)
=\arg\min_{x\in\mathbb{R}^n}\Big\{\tfrac12\|x-y\|_2^2+\lambda\|x\|_p^p\Big\}.
\]
By separability, computing $\prox_{\lambda\|\cdot\|_p^p}(y)$ reduces to the scalar problem for $y\in\mathbb{R}$,
\begin{equation}\label{eq:prox-scalar}
  \prox_{\lambda|\cdot|^p}(y) \;=\;\arg\min_{x\in\mathbb{R}}
  \Big\{ \tfrac12 (x-y)^2 + \lambda |x|^p \Big\}.
\end{equation}
Define the thresholds
\begin{equation}  
  \rho_{p,\lambda} := \big(2\lambda(1-p)\big)^{\frac{1}{2-p}}, \quad \tau_{p,\lambda} :=  \frac{2-p}{2(1-p)}\;\rho_{p,\lambda}. \label{def:tau-rho}
\end{equation}
According to \cite[Propositions 2 and 3]{chen2016computing},
\begin{equation}\label{eq:prox_p}
  \prox_{\lambda|\cdot|^p}(y) \;=\;
  \begin{cases}
    0, & |y| < \tau_{p,\lambda},\\[4pt]
    \{\,0,\; \operatorname{sgn}(y)\,\rho_{p,\lambda}\,\}, & |y| = \tau_{p,\lambda},\\[4pt]
    \operatorname{sgn}(y)\,x_+(|y|), & |y| > \tau_{p,\lambda},
  \end{cases}
\end{equation}
where $x_+(|y|)$ is the larger positive root of the first-order optimality condition for problem~\eqref{eq:prox-scalar}, that is,
\begin{equation}\label{def:Phi}
  \Phi_{p,\lambda,y}(x) \;:=\; x - y + \lambda p\,\operatorname{sgn}(x)\,|x|^{\,p-1} \;=\; 0. 
\end{equation}
It is known \cite[Lemma 3]{chen2016computing} that for $|y|>{\tau}_{p,\lambda}$, the function $\Phi_{p,\lambda,|y|}$ in \eqref{def:Phi}
has exactly two distinct roots $x_{-}(|y|)$ and $x_{+}(|y|)$ in $(0,\infty)$, ordered as 
\begin{equation}\label{eq:x+-}
    x_{-}(|y|)<\bigl(\lambda p(1-p)\bigr)^{\frac{1}{2-p}}<x_{+}(|y|).
\end{equation} 
Thus, proximal evaluation reduces to the larger root $x_+(|y|)$. Closed forms are known for $p=\tfrac12$ and $p=\tfrac23$ \cite{Cao-Sun-Xu:JVCIR:2013, chen2016computing}; otherwise Newton or bisection is commonly
used \cite{chen2016computing,liu2024bisection}. We seek instead a series valid for every $0<p<1$. 

\section{A Series Representation for
\texorpdfstring{$x_+(|y|)$}{the Larger Root}}\label{sec:series}

For $p=\tfrac12$ and $p=\tfrac23$, the stationarity equation~\eqref{def:Phi} reduces to a cubic or a structured quartic, respectively, yielding known closed-form solutions \cite{Cao-Sun-Xu:JVCIR:2013,chen2016computing}. For general $p$, however, the algebraic degree grows rapidly when $p$ is rational, while the equation is generally nonalgebraic when $p$ is irrational. We therefore reformulate~\eqref{def:Phi} as an inverse-function problem.

Writing the stationarity equation as $u=f(v)$, with $f$ analytic near a regular point, the Lagrange--B\"urmann theorem provides a convergent series for the inverse branch $v=g(u)$ \cite{flajolet2009analytic,henrici1974aca}. Applied here, it yields a unified series for $x_+(|y|)$ for every $p\in(0,1)$ and finite hypergeometric representations when $p$ is rational. Some proximity operators also admit closed forms through the Lambert \(W\) function---whose series and asymptotics follow from Lagrange inversion---notably for entropic/KL-type penalties and Poisson data terms \cite{Bauschke-Lindstrom:PAFA:2020,corless1996lambertw}. We first recall the Lagrange--B\"urmann theorem.

\begin{lemma}[Lagrange--B\"urmann \cite{Whittaker_Watson_1996}]\label{lem:lagrange-burmann}
Let $f$ be analytic in a neighborhood of $v_0\in\mathbb{C}$ with $f(v_0)=u_0$ and $f'(v_0)\neq 0$.
Then $f$ admits a local analytic inverse $g$ near $u_0$ with $g(u_0)=v_0$, and, for $u$ near $u_0$,
\begin{equation}\label{eq:LB-general}
  g(u)
  \;=\;
  v_0\;+\;\sum_{n=1}^\infty \frac{(u-u_0)^n}{n!}\,
  \left.\frac{d^{\,n-1}}{dw^{\,n-1}}
  \!\left(\left(\frac{w-v_0}{f(w)-u_0}\right)^{\!n}\right)\right|_{w=v_0}.
\end{equation}
\end{lemma}

Thus, if $f$ is locally one-to-one at $v_0$, its inverse can be expanded at $u_0=f(v_0)$ using derivatives of $f$, without solving $u=f(v)$ in closed form.

To apply the Lagrange--B\"urmann inversion theorem to the $\ell_p$ proximal problem, we recall the
first-order optimality condition \eqref{def:Phi}:
\[
\Phi_{p,\lambda,y}(x)\;:=\;x - y + \lambda p\,\operatorname{sgn}(x)\,|x|^{\,p-1}\;=\;0.
\]
Figure~\ref{fig:comparison_1} illustrates $\Phi_{p,\lambda,y}(x)$ for $p=\tfrac{2}{3}$, $\lambda=0.8$, and $y=1.5$; the qualitative shape is similar for other choices of $(p,\lambda,y)$. For $y>0$, any root of
$\Phi_{p,\lambda,y}$ lies in $(0,\infty)$ (in fact, from $x-y+\lambda p\,x^{p-1}=0$ we obtain
$x=y-\lambda p\,x^{p-1}\in(0,y)$), so we restrict attention to $x>0$.

To put the equation in an invertible analytic form suitable for Lagrange--B\"urmann, divide by $y>0$
and set
\[
\varepsilon \;:=\; \lambda p\,y^{\,p-2}, \qquad
v \;:=\; \frac{x}{y}, \qquad
f(v) \;:=\; (1-v)\,v^{\,1-p}.
\]
Then \eqref{def:Phi} is equivalent to the scalar equation
\[
f(v)\;=\;\varepsilon.
\]
Thus, solving $\Phi_{p,\lambda,y}(x)=0$ is equivalent to inverting the analytic map $v\mapsto f(v)$ at
the value $\varepsilon$, with the larger root corresponding to the branch $v\in(0,1)$ near $v=1$.
(Note that $f$ is analytic near $v=1$ and $f'(1)=-1\neq 0$, so the Lagrange--B\"urmann inversion
applies.)

On $(0,\infty)$, $\Phi_{p,\lambda,y}$ has a unique minimum, whereas $f$ has
the unique maximizer $v_\star=(1-p)/(2-p)$. When two stationary roots
exist, the right intersection $v_+\in(v_\star,1)$ maps to the larger root
$x_+(y)=yv_+$. Figure~\ref{fig:comparison_1} illustrates these equivalent
descriptions.

\begin{figure}[!t]
  \centering
  \begin{tabular}{@{}p{0.45\linewidth}p{0.45\linewidth}@{}}
    \includegraphics[width=\linewidth]{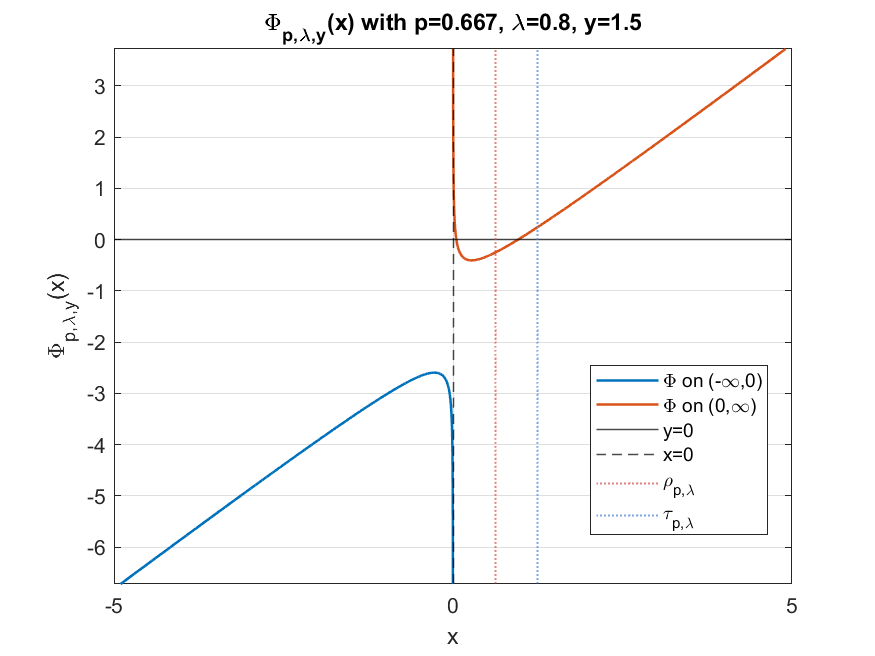} &
    \includegraphics[width=\linewidth]{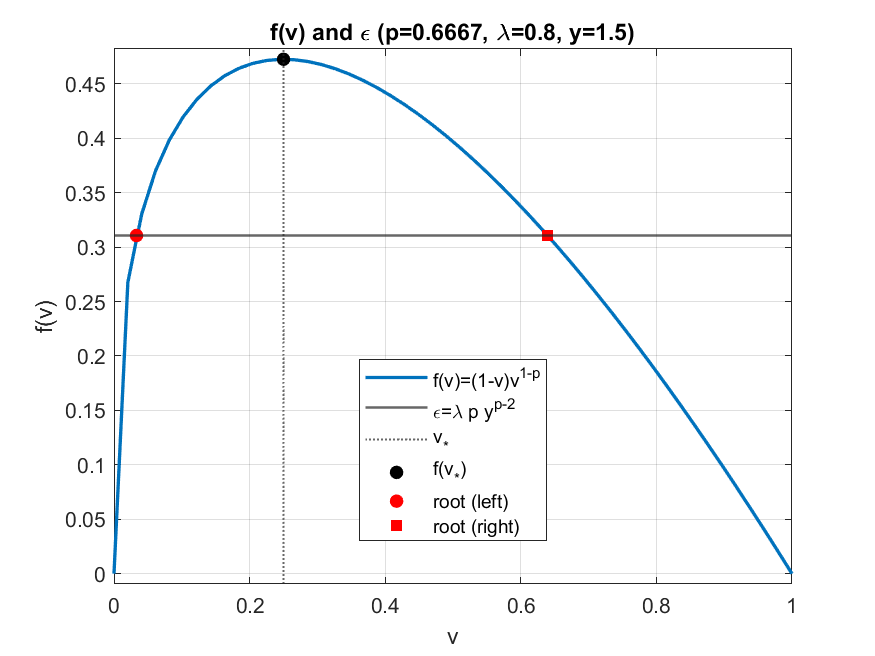} \\
    \centering (a) & \centering (b)
  \end{tabular}
  \caption{\textbf{Stationary equation and reparameterization.}
  (a) Graph of $\Phi_{p,\lambda,y}(x)$ for $p=\tfrac{2}{3}$, $\lambda=0.8$, and $y=1.5$.
  On $(0,\infty)$ the curve decreases to a unique minimum at
  $x_{\min}=(\lambda p(1-p))^{1/(2-p)}$ and then increases, producing two positive roots
  $x_-(y)<x_+(y)$ when the minimum is negative.
  (b) The transformed equation $f(v)=(1-v)v^{\,1-p}=\varepsilon$ with
  $\varepsilon=\lambda p\,y^{\,p-2}$ on $v\in[0,1]$. Intersections $v_\pm$ map to the zeros of
  $\Phi_{p,\lambda,y}$ by $x_\pm=y\,v_\pm$; the right intersection $v_+$ corresponds to the larger root
  $x_+(y)$ used in the proximal operator. The function $f$ attains its maximum at
  $v_\star=\frac{1-p}{2-p}$; the critical level $\varepsilon=f(v_\star)$ marks the transition between
  two, one (merged), and zero intersections.}
  \label{fig:comparison_1}
\end{figure}

\begin{lemma}[Local inverse and radius of convergence]\label{lem:f-inverse}
Let $0<p<1$. Set $\mathbb{D}:=\mathbb{C}\setminus(-\infty,0]$ and let $\Log:\mathbb{D}\to\mathbb{C}$ denote the principal branch with $\Log(1)=0$. Define
\[
f(v):=(1-v)v^{1-p},
\qquad
v^\alpha:=\exp\bigl(\alpha\Log v\bigr)
\quad (\alpha\in\mathbb{R}).
\]
Then there exist neighborhoods $U\ni0$ and $V\ni1$, with $V\subset\mathbb{D}$, such that $f:V\to U$ is biholomorphic and its inverse $g:U\to V$ satisfies $g(0)=1$. Moreover, the Taylor series of $g$ is
 \begin{equation}\label{eq:g}
    g(u)=1-u-\sum_{n=2}^{\infty}\frac{u^n}{n!}\cdot \prod_{j=0}^{n-2}\left( n(1-p) + j\right). 
 \end{equation}
and has radius of convergence $u_\star$, where
\begin{equation}\label{eq:vstar-ustar}
  v_{\star}:=\frac{1-p}{2-p}, \qquad
  u_{\star}:=f(v_{\star})=\frac{(1-p)^{\,1-p}}{(2-p)^{\,2-p}}.
\end{equation}
Here, $v_\star$ is the unique maximizer of $f$ on $(0,1)$ and $u_\star$ is its maximum value.
\end{lemma}

\begin{proof}
The function $f$ is analytic at $v_0=1$, with $u_0=f(v_0)=0$ and $f'(v_0)=-1\neq0$. By Lemma~\ref{lem:lagrange-burmann}, the local inverse $g$ of $f$ at $0$ satisfies
\begin{align*}
g(u) = 1 + \sum_{n=1}^{\infty}\frac{u^n}{n!}
 \left.\frac{d^{\,n-1}}{dw^{\,n-1}}
 \!\left(\left(\frac{w-1}{f(w)}\right)^{\!n}\right)\right|_{w=1}.
\end{align*}
For $f(w)=(1-w)w^{1-p}$,
$\big(\tfrac{w-1}{f(w)}\big)^{n}
=\big(-w^{\,p-1}\big)^{n}=(-1)^n w^{\,n(p-1)}$, hence
\begin{align*}
\left.\frac{d^{\,n-1}}{dw^{\,n-1}}
  \!\left(\left(\frac{w-1}{f(w)}\right)^{\!n}\right)\right|_{w=1}
  = -\prod_{j=0}^{n-2}\bigl(n(1-p)+j\bigr),
\end{align*}
for $n\ge2$. This proves \eqref{eq:g}.

To determine the radius, note first that $f'(v) = -v^{1-p} + (1-v)(1-p)v^{-p}$, so $f'$ vanishes only at $v_\star$ and $u_\star$ is a critical value of $f$.

Fix the open disk
\[
B:=\{v:|v-1|<\rho\},\qquad \rho:=1-v_\star.
\]
Then $f$ is analytic on $\bar{B}$ and $f'(v) \neq 0$ on $B$ (the only critical point $v_\star$ lies on $\partial B$). On $\partial B$, write $v=1-\rho e^{i \theta}$. A direct estimate gives 
\[
\min_{v \in \partial B} |f(v)|=\min_{\theta \in [0,2\pi]}\rho  |1-\rho e^{i\theta}|^{1-p}=\rho (1-\rho)^{1-p}=(1-v_\star)v_\star^{1-p}=u_\star. 
\]
Therefore, for any $|\tilde{u}|<u_\star$, Rouch\'e's theorem shows that $f$ and $f-\tilde{u}$ have the same number of zeros in $B$. Since $f$ has the unique zero $v=1$ in $B$, the equation $f(v)=\tilde{u}$ also has a unique solution $\tilde{v}\in B$. Because $f'(\tilde{v})\neq0$, the inverse $g$ is analytic at $\tilde{u}$, with $g(\tilde{u})=\tilde{v}$. Hence $g$ is analytic on $\{u:|u|<u_\star\}$.

Finally, $u_\star$ is a singular value for the inverse: the two real preimages
$v_\pm$ coalesce at $v_\star$, where
$f'(v_\star)=0$ and
$f''(v_\star)=-(2-p)v_\star^{-p}\neq0$. Thus the inverse has a square-root
branch point at $u_\star$, so the radius of convergence is exactly $u_\star$.
\end{proof}

Next we present a series representation and prove that it converges to the
larger stationary root throughout the active proximal domain.

\begin{theorem}[Series for the larger root]\label{thm:series-larger-root}
Let $0<p<1$, $\lambda>0$, and $y\ge\tau_{p,\lambda}$. Define 
\[
\begin{split}
\varepsilon(y)&:=\lambda p y^{p-2},\\
S(y)&:=y\left(1-\varepsilon(y)
-\sum_{n=2}^{\infty}\frac{\varepsilon(y)^n}{n!}
\prod_{j=0}^{n-2}\bigl(n(1-p)+j\bigr)\right).
\end{split}
\]
The series defining \(S(y)\) is absolutely convergent, uniformly for
\(y\in[\tau_{p,\lambda},\infty)\). Moreover, \(S(y)=x_+(y)\), the larger
positive solution of the stationary equation in \eqref{def:Phi}.

\end{theorem}
\begin{proof}
Let $f$ and its inverse branch $g$ be as in Lemma~\ref{lem:f-inverse}. Since $\varepsilon(y)$ decreases with $y$, for $y\ge\tau_{p,\lambda}$,
\begin{equation*}
    \varepsilon(y) \le \varepsilon(\tau_{p,\lambda})
    =q_pu_\star<u_\star,
    \qquad q_p:=p\,2^{1-p}<1.
\end{equation*}
Indeed, $p\mapsto\log(p\,2^{1-p})$ is strictly increasing on $(0,1)$ and
vanishes at $p=1$. The power series for $g$ therefore converges absolutely and
uniformly for all values of $\varepsilon(y)$ in the compact interval
$[0,q_pu_\star]$.
The same remains true after rescaling by \(y\). Indeed, for \(n\ge1\),
\[
 y\varepsilon(y)^n=(\lambda p)^n y^{\,1-n(2-p)}
\]
decreases on \([\tau_{p,\lambda},\infty)\). Thus the tail of the series for
\(S(y)\) is bounded by \(\tau_{p,\lambda}\) times the corresponding tail at
\(\varepsilon(\tau_{p,\lambda})\), which tends to zero uniformly.
Set $z(y):=g(\varepsilon(y))$ and $x(y):=y z(y)$. Since
$f(z(y))=\varepsilon(y)$, the definition of $f$ shows that
$\Phi_{p,\lambda,y}(x(y))=0$, and substituting \eqref{eq:g} gives the displayed formula for $S(y)$.

It remains to identify the branch. For each $0<\varepsilon<u_\star$, the equation $f(v)=\varepsilon$ has two real solutions, one in $(0,v_\star)$ and one in $(v_\star,1)$. Because $g(0)=1$, continuity gives $g(\varepsilon)\in(v_\star,1)$; hence it selects the larger solution. Therefore $x(y)=x_+(y)$ throughout the active domain.
\end{proof}

To conclude this section, we record a compact Gamma-function form of $S(y)$. Using $\Gamma(z+1)=z\Gamma(z)$, we may write
\begin{align}\label{eq:S(y)-Gamma}
S(y)=y\left(1-\sum_{n=1}^{\infty}\frac{\varepsilon(y)^n}{n!}
\frac{\Gamma(n(2-p)-1)}{\Gamma(n(1-p))}\right),
\end{align}
where $\varepsilon(y):=\lambda p y^{p-2}$. Denote the parenthesized quantity by $z(\varepsilon)$. Then $S(y)=y z(\varepsilon(y))$. 


\section{Rational Exponents and Hypergeometric Representations}\label{sec:special}

Let \(p=r/s\in(0,1)\), where \(r\) and \(s\) are coprime positive integers with \(r<s\). For \(x>0\), set \(t=x^{1/s}\). Multiplying the stationarity equation
$$
x-|y|+\lambda p x^{p-1}=0
$$
by \(x^{1-p}=t^{s-r}\) gives
$$
t^{2s-r}-|y|t^{s-r}+\lambda p=0,
$$
which is a trinomial algebraic equation. Thus, the existence of special-function representations in the rational case belongs to the general theory of trinomial roots \cite{glasser2000hypergeometric,Mane2026Trinomial,KocarTanabe2024Trinomial}. The purpose of this section is to specialize the unified inverse series derived above to the real proximal branch and give its residue-class decomposition explicitly. The resulting formula fixes both the relevant branch and its domain of validity. We then recover the familiar formulas for \(p=\tfrac12\) and \(p=\tfrac23\) and, to the best of our knowledge, obtain a compact \(p=\tfrac13\) representation not previously recorded in the proximal-operator literature.


Let us first introduce the rising Pochhammer symbol and the hypergeometric function that will simplify our discussion. 

We use
\[
(a)_n:=\overbrace{a(a+1)(a+2)\cdots(a+n-1)}^{n\ \text{factors}}
\]
for the rising Pochhammer symbol, with $(a)_0=1$; equivalently, $(a)_n=\Gamma(a+n)/\Gamma(a)$.

The generalized hypergeometric function associated with two sequences $(a_i)_{i=1}^\alpha$ and $(b_i)_{i=1}^\beta$ is defined by the power series
\[
{}_{\alpha}F_{\beta}(a_1,\ldots, a_\alpha;b_1,\ldots,b_\beta;z)=\sum_{n=0}^\infty \frac{(a_1)_n\cdots (a_\alpha)_n}{(b_1)_n\cdots (b_\beta)_n}\frac{z^n}{n!}. 
\]
The radius of convergence of ${}_\alpha F_\beta$ is as follows \cite{Luke1969}:
\begin{itemize}
\item If $\alpha\le \beta$, then ${}_\alpha F_\beta$ is entire (\emph{infinite} radius of convergence).
\item If $\alpha=\beta+1$, the radius is \emph{one} ($|z|<1$).
\item If $\alpha\ge \beta+2$, the radius is \emph{zero} (diverges for all $z\ne 0$, unless the series terminates).
\end{itemize}

We repeatedly use Gauss's multiplication formula \cite{abramowitz1972handbook},
\begin{equation}\label{eq:Multi-Theorem}
    \prod_{j=0}^{M-1} \Gamma\left(z+\frac{j}{M}\right)=(2\pi)^{\frac{M-1}{2}} M^{\frac{1}{2}-Mz} \Gamma(Mz).
\end{equation}

\begin{proposition}[Explicit rational-exponent decomposition]
\label{prop:rational-hypergeom}
Let \(p=r/s\in(0,1)\), where \(r\) and \(s\) are coprime, and set
\[
 A:=2s-r,\qquad B:=s-r,\qquad
 \kappa:=\frac{A^A}{B^B s^s}.
\]
For \(j=1,\ldots,s\), define
\[
 \alpha_j:=\frac{Aj}{s}-1,\qquad
 \beta_j:=\frac{Bj}{s},\qquad
 C_j:=\frac{\Gamma(\alpha_j)}{\Gamma(\beta_j)\Gamma(j+1)},
\]
and the parameter multisets
\begin{align*}
 \mathcal A_j&:=\left\{\frac{\alpha_j+k}{A}:k=0,\ldots,A-1\right\},\\
 \mathcal B_j&:=
 \left(\left\{\frac{\beta_j+k}{B}:k=0,\ldots,B-1\right\}\mathbin{\cup}\right.\\[-2pt]
 &\hspace{2.7em}\left.
 \left\{\frac{j+1+k}{s}:k=0,\ldots,s-1\right\}\right)
 \mathbin{\setminus}\{1\},
\end{align*}
where one occurrence of \(1\) is removed. Then, for \(|\eps|<u_\star\),
\begin{equation}\label{eq:general-rational-hypergeom}
 z(\eps)=1-\sum_{j=1}^{s} C_j\eps^j\,
 {}_{A}F_{A-1}\!\left(\mathcal A_j;\mathcal B_j;
 \kappa\eps^s\right).
\end{equation}
Moreover, \(\kappa u_\star^s=1\), so each hypergeometric argument has
modulus less than one on the disk of convergence of \(z\).
\end{proposition}

\begin{proof}
Write each \(n\ge1\) uniquely as \(n=sm+j\), with \(m\ge0\) and
\(j\in\{1,\ldots,s\}\). The coefficient in \eqref{eq:S(y)-Gamma} is
\[
 \frac{\Gamma(Am+\alpha_j)}
 {\Gamma(Bm+\beta_j)\Gamma(sm+j+1)}.
\]
Applying Gauss's multiplication formula~\eqref{eq:Multi-Theorem} to its
three Gamma factors and using
\(\Gamma(m+\theta)=\Gamma(\theta)(\theta)_m\) shows that its ratio to
\(C_j\) is
\[
 \kappa^m
 \frac{\prod_{a\in\mathcal A_j}(a)_m}
 {(1)_m\prod_{b\in\mathcal B_j}(b)_m}.
\]
Indeed, the two denominator parameter multisets contain exactly one
occurrence of \(1\), which supplies \((1)_m=m!\). Summing first over \(m\)
and then over \(j\) gives \eqref{eq:general-rational-hypergeom}. Finally,
\eqref{eq:vstar-ustar} gives
\(u_\star^s=B^B s^s/A^A=\kappa^{-1}\).
\end{proof}

Following this construction, we give three illustrative examples:
\(p=\tfrac12\), \(p=\tfrac23\), and \(p=\tfrac13\).

\subsection{The Case \texorpdfstring{$p=\frac{1}{2}$}{p=1/2}}

\begin{theorem}[Hypergeometric form for $p=\tfrac12$]\label{thm:S-Gauss}
For $\lambda>0$ and $y\ge\tau_{{1}/{2},\lambda}=\frac{3}{2}\lambda^{2/3}$, set $\varepsilon=\frac{1}{2}\lambda y^{-3/2}$ and $t=\frac{27}{4}\varepsilon^2$. Then the function $S(y)$ defined in Theorem~\ref{thm:series-larger-root} can be written as
\begin{equation}\label{eq:S-main-12}
  S(y)
  \;=\;
  y\!\left(
  \frac{2}{3}
  +\frac{1}{3}\,{}_2F_1\!\left(-\frac{1}{3},\frac{1}{3};\frac{1}{2};\,t\right)
  -\varepsilon\;{}_2F_1\!\left(\frac{1}{6},\frac{5}{6};\frac{3}{2};\,t\right)
  \right).
\end{equation}
\end{theorem}
\begin{proof}
For \(p=\tfrac12\), write

\[
 a_n:=\frac{\Gamma(3n/2-1)}{n!\,\Gamma(n/2)},
 \qquad \frac{S(y)}{y}=1-\sum_{n\ge1}a_n\varepsilon^n.
\]
Splitting the sum into even and odd indices and applying
\eqref{eq:Multi-Theorem} gives
\[
 a_{2m}=-\frac13\left(\frac{27}{4}\right)^m
 \frac{(-\frac13)_m(\frac13)_m}{(\frac12)_m m!},\qquad
 a_{2m+1}=\left(\frac{27}{4}\right)^m
 \frac{(\frac16)_m(\frac56)_m}{(\frac32)_m m!}.
\]
The two identities hold for \(m\ge1\) and \(m\ge0\), respectively.
Substitution and \(t=27\varepsilon^2/4\) yield
\eqref{eq:S-main-12}. Finally,
\(y\ge\tau_{1/2,\lambda}\) implies
\(\varepsilon\le\sqrt2/(3\sqrt3)\), hence \(0\le t\le\tfrac12\),
so both Gauss series converge absolutely.
\end{proof}

The standard identities for
\({}_2F_1(-a,a;\tfrac12;\sin^2\phi)\) and
\({}_2F_1(a,1-a;\tfrac32;\sin^2\phi)\)
\cite[Eqs.~15.1.15 and 15.1.17]{abramowitz1972handbook}, applied with
\(\phi=\arcsin\sqrt t\), reduce \eqref{eq:S-main-12} to the following
trigonometric form.
\begin{corollary}\label{cor:S-Gauss}
Under the conditions of Theorem~\ref{thm:S-Gauss}, set
\(\phi=\arcsin\sqrt{t}\). Then
\begin{equation}\label{eq:S-main-Phi}
  S(y)
  \;=\;
  y\!\left(
  \frac{2}{3}+\frac{1}{3}\cos\!\left(\tfrac{2}{3}\phi\right)
  -\frac{\sqrt{3}}{3}\sin\!\left(\tfrac{2}{3}\phi\right)\right)=\frac{2y}{3}\left(1+\cos\!\left(\frac{2\phi}{3}+\frac{\pi}{3}\right)\right).
\end{equation}
\end{corollary}

Consequently, for \(y\ge\tau_{1/2,\lambda}\), the larger positive root
\(x_+(y)\) coincides with the expression for \(S(y)\) in
Corollary~\ref{cor:S-Gauss}.

\subsection{The Cases \texorpdfstring{$p=\frac{2}{3}$ and $p=\frac{1}{3}$}{p=2/3 and p=1/3}}
\begin{theorem}[Hypergeometric form for $p=\tfrac23$]\label{thm:S-Gauss-23}
For $\lambda>0$ and
$y\ge\tau_{{2}/{3},\lambda}=2(\frac{2}{3}\lambda)^{3/4}$, set
\[
\varepsilon=\frac{2}{3}\lambda y^{-4/3},\qquad
t=\frac{256}{27}\varepsilon^3,
\]
and let $S(y)$ be as in Theorem~\ref{thm:series-larger-root}. Then
\begin{align}\label{eq:S-main}
\frac{S(y)}{y}
={}&\frac{3}{4}
+\frac{1}{4}\,{}_3F_2\!\left(\frac{1}{2},-\frac{1}{4},\frac{1}{4};\frac{1}{3},\frac{2}{3};\,t\right)\nonumber\\
&-\varepsilon\,{}_3F_2\!\left(\frac{1}{12},\frac{7}{12},\frac{5}{6};\frac{2}{3},\frac{4}{3};\,t\right)
-\frac{\varepsilon^2}{3}\,{}_3F_2\!\left(\frac{5}{12},\frac{11}{12},\frac{7}{6};\frac{4}{3},\frac{5}{3};\,t\right).
\end{align}
and the series converges absolutely for \(0\le t\le16/27\) whenever
\(y\ge\tau_{{2}/{3},\lambda}\).
\end{theorem}

\begin{theorem}[Hypergeometric form for $p=\tfrac13$]\label{thm:main13}
Let $\lambda>0$ and $y\ge\tau_{1/3,\lambda}$. Set
\begin{equation}\label{eq:eps_t_def}
\eps=\tfrac13\,\lambda\,y^{-5/3},\qquad t=\frac{3125}{108}\,\eps^{3}.
\end{equation}
Then
\begin{align}\label{eq:S13-hyp}
\frac{S(y)}{y}
={}&\frac{3}{5}
+\frac{2}{5}\,{}_4F_3\!\Big(\!-\tfrac15,\tfrac15,\tfrac25,\tfrac35;\;\tfrac13,\tfrac12,\tfrac23;\,t\Big)\nonumber\\
&-\eps\,{}_4F_3\!\Big(\tfrac{2}{15},\tfrac{8}{15},\tfrac{11}{15},\tfrac{14}{15};\;\tfrac23,\tfrac56,\tfrac43;\,t\Big)\nonumber\\
&-\frac{2}{3}\,\eps^2\,{}_4F_3\!\Big(\tfrac{7}{15},\tfrac{13}{15},\tfrac{16}{15},\tfrac{19}{15};\;\tfrac76,\tfrac43,\tfrac53;\,t\Big).
\end{align}
Moreover, at the threshold $y=\tau_{1/3,\lambda}$ one has $t=\tfrac{4}{27}<1$, so for all $y\ge \tau_{1/3,\lambda}$ the argument $t\in(0,\tfrac{4}{27}]$ and all ${}_4F_3$ series converge absolutely.
\end{theorem}

\begin{proof}
For both theorems, write
\[
a_n:=\frac{\Gamma(n(2-p)-1)}{n!\,\Gamma(n(1-p))},
\qquad z(\eps)=1-\sum_{n\ge1}a_n\eps^n.
\]
For $p=\tfrac23$, split the sum according to $n=3m$, $3m+1$, and $3m+2$. Gauss's multiplication formula gives, for $m\ge0$ (and for $m\ge1$ in the first line),
\begin{align*}
-a_{3m}&=\frac14\frac{(\frac12)_m(-\frac14)_m(\frac14)_m}
 {(\frac13)_m(\frac23)_m m!}\left(\frac{256}{27}\right)^m,\\
a_{3m+1}&=\frac{(\frac1{12})_m(\frac7{12})_m(\frac56)_m}
 {(\frac23)_m(\frac43)_m m!}\left(\frac{256}{27}\right)^m,\\
a_{3m+2}&=\frac13\frac{(\frac5{12})_m(\frac{11}{12})_m(\frac76)_m}
 {(\frac43)_m(\frac53)_m m!}\left(\frac{256}{27}\right)^m.
\end{align*}
Substitution into the three residue-class sums proves Theorem~\ref{thm:S-Gauss-23}.

For $p=\tfrac13$, the same calculation yields
\begin{align*}
-a_{3m}&=\frac25\frac{(-\frac15)_m(\frac15)_m(\frac25)_m(\frac35)_m}
 {(\frac13)_m(\frac12)_m(\frac23)_m m!}\left(\frac{3125}{108}\right)^m,\\
a_{3m+1}&=\frac{(\frac2{15})_m(\frac8{15})_m(\frac{11}{15})_m(\frac{14}{15})_m}
 {(\frac23)_m(\frac56)_m(\frac43)_m m!}\left(\frac{3125}{108}\right)^m,\\
a_{3m+2}&=\frac23\frac{(\frac7{15})_m(\frac{13}{15})_m(\frac{16}{15})_m(\frac{19}{15})_m}
 {(\frac76)_m(\frac43)_m(\frac53)_m m!}\left(\frac{3125}{108}\right)^m.
\end{align*}
These identities give \eqref{eq:S13-hyp}. The stated ranges follow by substituting the activation thresholds into $t$; all displayed hypergeometric series have radius one. The proof is complete.
\end{proof}

MATLAB's built-in function \texttt{hypergeom} evaluates the generalized hypergeometric functions ${}_pF_q$ and can be used to compute the expressions above.

Although these special-function representations exist for all rational exponents, formulas by radicals are much more restrictive. A separate supplement records the corresponding Galois-theoretic classification.

\section{Mellin--Barnes Representation and Error-Controlled Evaluation}\label{sec:MB}

In this section, we recast the larger proximal root through the normalized map $\varepsilon(y):=\lambda p\,y^{p-2}$ and $z(\varepsilon):=x_+(y)/y$. We derive a Mellin--Barnes (MB) integral for $z$, show that a contour shift reproduces a finite part of the Lagrange--B\"urmann series, and bound the remaining vertical tail. The mathematical truncation error and the numerical quadrature error are recorded separately. This distinction is essential: the asymptotic estimates below prove uniform analytic bounds, whereas a formally certified floating-point implementation would additionally require explicit Gamma-function constants and validated quadrature.

\subsection{Mellin--Barnes Representation and Equivalence to the Series}

In this subsection, we derive a Mellin--Barnes (MB) integral for $z(\varepsilon)$ and show that closing the contour to the right recovers the Lagrange--B\"urmann series.

\begin{lemma}[Semicircle sine bound]\label{lem:sin-strip}
Let $0<p<1$ and $\gamma\in\big(\tfrac{1}{2-p},\,1\big)$. For $R>0$ and $\theta\in[-\pi/2,\pi/2]$, set
\[
s=\gamma+Re^{i\theta}=\Re s+i\, \Im s,\qquad \Re s=\gamma+R\cos\theta,\quad \Im s=R\sin\theta.
\]
If $R=R_k:=k+\tfrac12-\gamma$ with $k\in\mathbb{N}$ sufficiently large, then
\[
\frac{1}{|\sin(\pi s)|}\;\le\; 
\left\{
  \begin{array}{ll}
    2, & \text{if $|\theta| \le R^{-1}$}, \\
    \frac{2}{1 - e^{-4}}e^{-\pi |\Im s|}, & \text{if $R^{-1}\le |\theta| \le \frac{\pi}{2}$}.
  \end{array}
\right.
\]
\end{lemma}

\begin{proof} Use the identity $|\sin(\pi s)|^2=\sin^2(\pi \,\Re s)+\sinh^2(\pi \, \Im s)$, fix large $k$, and write $R=R_k$. We consider separately $|\theta|\le R^{-1}$ and $R^{-1}\le |\theta|\le\pi/2$.

\emph{Small angles.}
Since $\cos\theta\ge 1-\theta^2/2$, we have
\[
\Re s=\gamma+R\cos\theta\ \ge\ \gamma+R\Bigl(1-\frac{\theta^2}{2}\Bigr)
=(k+\tfrac12)-\frac{R\theta^2}{2}.
\]
Let $\delta:=(k+\tfrac12)-\Re s\in\big[0,\frac{R\theta^2}{2}\big]$. If $|\theta|\le R^{-1}$ then
$\delta\le \frac{R\theta^{2}}{2} \le \tfrac{R}{2}\,R^{-2}=\tfrac{1}{2R}$. Hence
\[
|\sin(\pi\Re s)|
=\bigl|\sin\bigl(\pi(k+\tfrac12-\delta)\bigr)\bigr|
=\cos(\pi\delta)
\ge\cos\Bigl(\frac{\pi}{2R}\Bigr)\ge\frac12
\]
for all large \(R\). Therefore,
\[
 |\sin(\pi s)|\ge|\sin(\pi\Re s)|\ge\tfrac12,
 \qquad \frac1{|\sin(\pi s)|}\le2,
\]
in the first range.

\emph{Remaining angles.} Using $\sin |\theta| \ge \frac{2}{\pi}|\theta| $ when $|\theta| \le \frac{\pi}{2}$, 
\[
    |\Im s| = R |\sin \theta| \ge R\cdot \frac{2}{\pi}|\theta| \ge \frac{2}{\pi}.
\]
Consequently, 
\[
\sinh(\pi|\Im s|)
=\frac{e^{\pi|\Im s|}}{2}\bigl(1-e^{-2\pi|\Im s|}\bigr)
\ge\frac{1-e^{-4}}{2}e^{\pi|\Im s|}.
\]
Hence, $\frac{1}{\sinh \pi|\Im s|} \le \frac{2}{1 - e^{-4}}e^{-\pi |\Im s|}$. Combining the two cases yields the claim.
\end{proof}

\begin{lemma}[Uniform bound for $G(s)$]\label{lem:uniform-H}
Fix $0<p<1$ and $\gamma\in\big(\tfrac{1}{2-p},\,1\big)$. Define
\[
G(s):=\frac{\Gamma\big((2-p)s-1\big)}{\Gamma(1+s)\,\Gamma\big((1-p)s\big)},\qquad 
A:=(2-p)\log(2-p)-(1-p)\log(1-p).
\]
For the right semicircle
\[
\Sigma_R:=\{\ \gamma+R\cos\theta+iR\sin\theta:\ \theta\in[-\tfrac{\pi}{2},\tfrac{\pi}{2}]\ \},
\]
there exist constants $C>0$ and $R_0>0$ (depending only on $p,\gamma$) such that for all $R\ge R_0$ and all $s\in\Sigma_R$,
\begin{equation}\label{eq:H-bound}
|G(s)|\ \le\ C\, R^{-3/2}\,e^{A\,\Re s}.
\end{equation}
\end{lemma}

\begin{proof}
Stirling's formula, uniformly for the three Gamma-function arguments on the
right half-plane under consideration, gives
\[
 G(s)=C_p e^{As}s^{-3/2}\bigl(1+O(|s|^{-1})\bigr)
 \qquad (|s|\to\infty),
\]
where \(C_p\ne0\) depends only on \(p\). On \(\Sigma_R\),
\[
 |s|=|\gamma+Re^{i\theta}|\ge R-\gamma\ge R/2
\]
once \(R\ge2\gamma\). Taking absolute values and enlarging a constant to cover
the bounded remaining range proves \eqref{eq:H-bound}.
\end{proof}

Armed with the two lemmas above, we are ready to establish the main result of this subsection. 

\begin{theorem}[Mellin--Barnes representation]\label{thm:MB}
Let \(0<p<1\) and
\(0<\varepsilon<u_{\star}:=(1-p)^{1-p}/(2-p)^{2-p}\). Define
\[
z(\varepsilon):=1-\sum_{n=1}^\infty \frac{\varepsilon^{\,n}}{n!}\,
\frac{\Gamma\big(n(2-p)-1\big)}{\Gamma\big(n(1-p)\big)}.
\]
Then for any $\gamma\in\big(\tfrac{1}{2-p},\,1\big)$,
\begin{equation}\label{eq:MB}
z(\varepsilon)
=1-\frac{1}{2\pi i}\int_{\gamma-i\infty}^{\gamma+i\infty}
\Gamma(-s)\,\frac{\Gamma\big(s(2-p)-1\big)}{\Gamma\big(s(1-p)\big)}\,
(-\varepsilon)^{\,s}\,ds,
\end{equation}
where
\[
 (-\varepsilon)^s:=\exp\bigl(s(\log\varepsilon+i\pi)\bigr)
\]
denotes the upper boundary value on the negative real axis.
\end{theorem}

\begin{proof}
Let
\[
f(s):=\Gamma(-s)\,\frac{\Gamma\big(s(2-p)-1\big)}{\Gamma\big(s(1-p)\big)}\,(-\varepsilon)^{\,s}.
\]
The poles of (f) arise from the following two gamma factors:
\[
\begin{aligned}
\Gamma(-s)&:&&s=0,1,2,\ldots,\\
\Gamma\bigl(s(2-p)-1\bigr)&:&&
s=\frac{1}{2-p},0,-\frac{1}{2-p},-\frac{2}{2-p},\ldots.
\end{aligned}
\]
Fix $\gamma\in\big(\tfrac{1}{2-p},1\big)$ and let $L_{\gamma,R}$ be the vertical segment
\[
    L_{\gamma,R} := \{\gamma+it: -R\le t \le R\}.
\]
For $R=R_k$ as in Lemma~\ref{lem:sin-strip}, let
\[
\Sigma_R=\{\ \gamma+R\cos\theta+iR\sin\theta:\ \theta\in[\tfrac{\pi}{2},-\tfrac{\pi}{2}]\,\}.
\]
Orient \(L_{\gamma,R}\) from bottom to top. Then
\(L_{\gamma,R}\cup\Sigma_R\) is clockwise and encloses the simple poles
\(s=1,2,\dots,\lfloor\gamma+R\rfloor\) of \(\Gamma(-s)\), but neither \(0\)
nor \(1/(2-p)\).

At $s=n\in\{1,2,\dots\}$, $\operatorname{Res}_{s=n}\Gamma(-s)=(-1)^{n+1}/n!$, and the other factors are analytic,
so
\[
\operatorname{Res}_{s=n} f(s)
=-\,\frac{\varepsilon^{\,n}}{n!}\,\frac{\Gamma\big(n(2-p)-1\big)}{\Gamma\big(n(1-p)\big)}.
\]
Because the contour is clockwise, the residue theorem introduces a minus sign. Combining it with the negative residues above gives
\begin{equation}\label{eq:res-sum}
\frac{1}{2\pi i}\int_{L_{\gamma,R}\cup\Sigma_R} f(s)\,ds
=\sum_{n=1}^{\lfloor\gamma+R\rfloor} \frac{\varepsilon^{\,n}}{n!}\,
\frac{\Gamma\big(n(2-p)-1\big)}{\Gamma\big(n(1-p)\big)}.
\end{equation}

Here we write $s = \Re s + i\Im s = (\gamma+R\cos \theta)+ iR\sin \theta$. Using the reflection formula $\Gamma(-s)\Gamma(1+s) = -\frac{\pi}{\sin \pi s}$ gives 
\[
    f(s) = -\frac{\pi}{\sin \pi s}\frac{\Gamma(s(2-p)-1)}{\Gamma(1+s)\,\Gamma(s(1-p))}(-\varepsilon)^{s}.
\]
For simplicity, let $G(s):=\frac{\Gamma(s(2-p)-1)}{\Gamma(1+s)\,\Gamma(s(1-p))}$. By Lemma~\ref{lem:uniform-H} there exists a real constant $C > 0$ such that
\[
    \left|G(s)\right|\le C R^{-\frac{3}{2}}e^{A\Re s}. 
\]
Moreover,
\[
|(-\varepsilon)^s| = \big|e^{s(\log\varepsilon+i\pi)}\big|
=e^{\Re s\log \varepsilon-\pi \Im s}.
\]

We now combine these estimates with Lemma~\ref{lem:sin-strip}. Set $A:=(2-p)\log(2-p)-(1-p)\log(1-p)=-\log u_{\star}$ and $\widetilde A:=A+\log\varepsilon=\log(\varepsilon/u_\star)<0$.

\paragraph{Case (a): $|\theta| \le R^{-1}$. } In this case, $|\Im s|\le1$, and hence
\begin{align*}
    \left| \frac{\pi}{\sin \pi s}\cdot G(s)\cdot (-\varepsilon)^{s}\right|
    &\le (2\pi)\cdot \left(C R^{-\frac{3}{2}} e^{A\Re s} \right) \cdot \left( e^{\Re s\log \eps -\pi \Im s}\right) \\
    &\le (2\pi C e^{\pi+\widetilde A\gamma}) R^{-\frac{3}{2}}
\end{align*}
where $\Re s=\gamma+R\cos\theta\ge\gamma$ and $e^{-\pi\Im s}\le e^\pi$ were used.

\paragraph{Case (b): $R^{-1}\le \theta \le \frac{\pi}{2}$.} 
\begin{align*}
    \left| \frac{\pi}{\sin \pi s}\cdot G(s)\cdot (-\varepsilon)^{s}\right| &\le \left(\frac{2\pi}{1 - e^{-4}}\ e^{-\pi |\Im s|}\right)\cdot (C R^{-\frac{3}{2}} e^{A\Re s}) \cdot (e^{\Re s\log \eps - \pi \Im s}) \\
    &\le \left(\frac{2\pi C e^{\widetilde A\gamma -4}}{1 - e^{-4}}\right) R^{-\frac{3}{2}}
\end{align*}
where we apply $\Re s = \gamma + R\cos \theta \ge \gamma$ and $\Im s = R\sin \theta \ge R(\frac{2}{\pi}\theta) \ge \frac{2}{\pi}$. 

\paragraph{Case (c): $-\frac{\pi}{2}\le \theta \le -R^{-1}$.} 
\begin{align*}
    \left| \frac{\pi}{\sin \pi s}\cdot G(s)\cdot (-\varepsilon)^{s}\right| &\le \left(\frac{2\pi C}{1 - e^{-4}}\right)\cdot (e^{-\pi |\Im s| -\pi \Im s}) \cdot R^{-\frac{3}{2}}\cdot e^{\widetilde A\Re s} \\
    &\le \left( \frac{2\pi C  e^{\widetilde A\gamma}}{1 - e^{-4}}\right) R^{-\frac{3}{2}}.
\end{align*}
Here we use $\Re s = \gamma + R\cos \theta \ge \gamma$.

Since $\widetilde A<0$, there is a constant $C>0$ such that, for every $|\theta|\le\frac{\pi}{2}$,
\[
    |f(s)| \le \left|\frac{\pi}{\sin \pi s}\right| \cdot |G(s)| \cdot |(-\varepsilon)^{s}| \le C R^{-\frac{3}{2}}.
\]

Thus we have 
\begin{align*}
    \left|\int_{\Sigma_{R}}f(s)ds \right| \le \int_{\Sigma_{R}}|f(s)| |ds| \le C R^{-\frac{3}{2}} (\pi R) = C R^{-\frac{1}{2}} \longrightarrow 0, \qquad (R \to \infty)
\end{align*}

Letting $R\to\infty$ along the chosen sequence in \eqref{eq:res-sum} gives
\[
\frac{1}{2\pi i}\int_{\gamma-i\infty}^{\gamma+i\infty} f(s)\,ds
=\sum_{n=1}^{\infty} \frac{\varepsilon^{\,n}}{n!}\,
\frac{\Gamma\big(n(2-p)-1\big)}{\Gamma\big(n(1-p)\big)}.
\]
Subtracting this identity from \(1\) gives \eqref{eq:MB}.
\end{proof}

\subsection{A Residue-Corrected Error-Controlled Evaluator}

We combine a finite residue sum with a truncated vertical MB segment of height
$2T$. The resulting representation can reduce finite-order truncation error
in some regimes, although its numerical benefit depends on \(p\) and the
contour parameters. We first bound the analytic truncation in exact arithmetic
and then add the quadrature error incurred by an implementation.

For \(0<p<1\) and \(0<\varepsilon<u_\star\), define
\[
 a_n:=\frac{1}{n!}\,\frac{\Gamma\big((2-p)n-1\big)}{\Gamma\big((1-p)n\big)},\qquad
 H(s):=-\frac{\Gamma(-s)\,\Gamma\big(s(2-p)-1\big)}
 {\Gamma\big(s(1-p)\big)}.
\]
Then for any $\gamma\in\big(\tfrac{1}{2-p},1\big)$, by Theorem~\ref{thm:MB}
\begin{equation}\label{eq:MB-either}
 z(\varepsilon)=1-\sum_{n=1}^{\infty} a_n\,\varepsilon^n
 \;=\; 1+\frac{1}{2\pi i}\int_{\gamma-i\infty}^{\gamma+i\infty} H(s)\,(-\varepsilon)^s\,ds,
\end{equation}
where \((-\varepsilon)^s=\exp(s(\log\varepsilon+i\pi))\) is the upper
boundary value.

Fix an integer $N\ge 1$ and choose a shifted vertical line $\Re s=\sigma_N$ with $\sigma_N\in(N,N+1)$. Shifting the contour in \eqref{eq:MB-either} to $\Re s=\sigma_N$ and summing the residues at $s=1,\dots,N$ yields the \emph{exact} identity
\begin{equation}\label{eq:hybrid-exact}
 z(\varepsilon)
 = 1 - \sum_{n=1}^{N} a_n\,\varepsilon^n
 + \frac{1}{2\pi i}\int_{\sigma_N-i\infty}^{\sigma_N+i\infty} H(s)\,(-\varepsilon)^s\,ds.
\end{equation}
Define the truncated map by cutting the vertical segment at height $T>0$:
\begin{equation}\label{eq:trunc-map-shifted}
 \widetilde z^{\,\sigma_N}_{N,T}(\varepsilon)
 := 1 - \sum_{n=1}^{N} a_n\,\varepsilon^n
 + \frac{1}{2\pi i}\int_{\sigma_N-iT}^{\sigma_N+iT} H(s)\,(-\varepsilon)^s\,ds.
\end{equation}
Subtracting \eqref{eq:trunc-map-shifted} from \eqref{eq:hybrid-exact} gives the clean error split
\begin{equation}\label{eq:error-split}
 z(\varepsilon)-\widetilde z^{\,\sigma_N}_{N,T}(\varepsilon)
 = \underbrace{\frac{1}{2\pi i}\!\left(\int_{\sigma_N+iT}^{\sigma_N+i\infty}+\int_{\sigma_N-i\infty}^{\sigma_N-iT}\right)\! H(s)\,(-\varepsilon)^s\,ds}_{E_{\mathrm{MB}}(T,\varepsilon)}.
\end{equation}

\begin{lemma}[Series-tail bound]\label{lem:series}
Set \(E_{\mathrm{series}}(N,\varepsilon)
:=-\sum_{n>N}a_n\varepsilon^n\). There exists \(C_2>0\), depending only on
\(p\), such that for every \(0<r<u_\star\),
\begin{equation}\label{eq:series-tail}
\sup_{|\varepsilon|\le r} \big|E_{\mathrm{series}}(N,\varepsilon)\big|
\;\le\; C_2\,\frac{(r/u_\star)^{N+1}}{1-r/u_\star}\xrightarrow[N\to\infty]{} 0.
\end{equation}
\end{lemma}

\begin{proof}
Stirling's formula, applied to the positive real Gamma-function arguments, gives $a_n=O\big(n^{-3/2}u_\star^{-n}\big)$. Hence $|a_n|\le C_2u_\star^{-n}$ for some $C_2>0$ and every $n\ge1$, after enlarging $C_2$ to cover finitely many small indices. Thus, for $|\varepsilon|\le r$,
\(
\sum_{n>N}|a_n|\,|\varepsilon|^n
\le C_2 \sum_{n>N}(r/u_\star)^n
= C_2\,\frac{(r/u_\star)^{N+1}}{1-r/u_\star}.
\)
\end{proof}

\begin{lemma}[Vertical-tail (MB) bound]\label{lem:MBtail}
Fix $\sigma_{N}\in\big(N,\,N+1\big)$ with $N \ge 1$. For any $0<r<u_{\star}$ there exists a constant
$C' = C'(p,\sigma_{N},r)$ such that for all $T>0$,
\[
\sup_{0<\varepsilon\le r}\, |E_{\mathrm{MB}}(T,\varepsilon)|
\;\le\; 2\,C'\,T^{-1/2}\;\xrightarrow[T\to\infty]{}\;0.
\]
\end{lemma}

\begin{proof}
The standard vertical-line Gamma estimate \cite[Eq.~5.11.9]{NISTHandbook2010}
\[
|\Gamma(z)|\le C\sqrt{2\pi}\,|\Im z|^{\Re z-1/2}
e^{-\pi|\Im z|/2},\qquad |\Im z|\to\infty,
\]
gives, after enlarging the constant to cover bounded nonzero $t$, some $C=C(p,\sigma_N)>0$ such that
\[
|H(\sigma_{N}+it)| \;\le\; C\,|t|^{-3/2}\,e^{A\sigma_N}\,e^{-\pi|t|},
\qquad
A := -\log u_{\star}.
\]
Using the prescribed upper boundary value, we have
\[
|(-\varepsilon)^{\sigma_N+it}|
=\varepsilon^{\sigma_N}e^{-\pi t}
\le\varepsilon^{\sigma_N}e^{\pi|t|}.
\]
Hence, for \(0<\varepsilon\le r\),
\[
\begin{aligned}
|E_{\mathrm{MB}}(T,\varepsilon)|
&= \frac{1}{2\pi}\Big|\int_{|t|>T} H(\sigma_{N}+it)\,(-\varepsilon)^{\sigma_{N}+it}\,dt\Big| \\
&\le \frac{C e^{A\sigma_{N}} r^{\sigma_N}}{2\pi}
 \int_{|t|>T} |t|^{-3/2}\,dt\\
&= \frac{2C e^{A\sigma_{N}}r^{\sigma_{N}}}{\pi}T^{-1/2}
=2C'T^{-1/2},
\end{aligned}
\]
with $C' := \frac{C e^{A\sigma_{N}} r^{\sigma_{N}}}{\pi}$, which depends only on $(p,\sigma_{N},r)$.
\end{proof}

The main result of this subsection is as follows. 

\begin{theorem}[Two-regime analytic truncation]\label{thm:theorem54}
Fix \(0<p<1\), \(\lambda>0\), a target accuracy \(\eta>0\), a cap
\(R\in(0,u_\star)\), and any \(\rho\in(0,R)\). Then there exist an integer
\(N\ge1\), a shift \(\sigma_N\in(N,N+1)\), and a height \(T>0\) such that
the piecewise map
\begin{equation}\label{eq:z-hat}
\hat z(\varepsilon):=\begin{cases}
\displaystyle 1-\sum_{n=1}^{N} a_n\,\varepsilon^n, & 0\le\varepsilon\le \rho,\\[0.6em]
\displaystyle 1-\sum_{n=1}^{N} a_n\,\varepsilon^n + \dfrac{1}{2\pi i}\int_{\sigma_N-iT}^{\sigma_N+iT} H(s)\,(-\varepsilon)^s\,ds, & \rho<\varepsilon\le R,
\end{cases}
\end{equation}
satisfies \(\sup_{0\le\varepsilon\le R}
|z(\varepsilon)-\hat z(\varepsilon)|\le\eta\).
Consequently, for \(\varepsilon(y)=\lambda p y^{p-2}\) and
\(y\ge\tau_{p,\lambda}\) with \(\varepsilon(y)\le R\), the approximation
\(\hat x_+(y):=y\hat z(\varepsilon(y))\) obeys
\(|\hat x_+(y)-x_+(y)|\le y\eta\).
In particular, choosing \(q_pu_\star\le R<u_\star\) covers the full active
proximal domain.
\end{theorem}

\begin{proof}
\textit{(i) Small-\(\varepsilon\) branch.} For \(0\le\varepsilon\le\rho\),
\(\hat z(\varepsilon)=1-\sum_{n=1}^{N}a_n\varepsilon^n\). Let \(N\) be
large enough that
\begin{equation} \label{eq:chooseN}
    |z(\varepsilon) - \hat{z}(\varepsilon)| \le \sum_{n > N}|a_{n}|\cdot |\eps|^{n} \le C_{2} \frac{(\rho/u_\star)^{N+1}}{1 - \rho/u_\star} \le \eta
\end{equation}
which is possible by Lemma~\ref{lem:series}. 

\smallskip
\textit{(ii) MB-augmented branch.} For any \(\rho<\varepsilon\le R\),
\begin{align*}
    \hat{z}(\eps) = 1-\sum_{n=1}^{N} a_n\,\varepsilon^n + \dfrac{1}{2\pi i}\int_{\sigma_N-iT}^{\sigma_N+iT} H(s)\,(-\varepsilon)^s\,ds.
\end{align*}

Use the exact split~\eqref{eq:error-split} on $\Re s=\sigma_N\in(N,N+1)$.
Choose $T>0$ so that $2C' T^{-1/2}\le \eta$ via Lemma~\ref{lem:MBtail} (with $\sigma=\sigma_N$ and $r=R$).
Then for any \(\rho<\varepsilon\le R\),
\(
|z(\varepsilon)-\hat z(\varepsilon)|
\le 2C' T^{-1/2}\le \eta.
\)

\smallskip
\textit{(iii) Transfer to $x_+(y)$.} Set $\varepsilon(y)=\lambda p y^{p-2}$ and $\hat x_+(y):=y\,\hat z(\varepsilon(y))$.
Then for any $y$ with $|\varepsilon(y)|\le R$,
\(
|\hat x_+(y)-x_+(y)|
= y\,|\hat z(\varepsilon(y))-z(\varepsilon(y))|\le y\,\eta.
\)
\end{proof}

For numerical use, let $I_{N,T}(\varepsilon)$ denote the finite integral in the second branch of \eqref{eq:z-hat}. If a validated quadrature returns $\widehat I_{N,T}(\varepsilon)$ and a bound $E_{\rm quad}(\varepsilon)$ satisfying
\[
|I_{N,T}(\varepsilon)-\widehat I_{N,T}(\varepsilon)|\le E_{\rm quad}(\varepsilon),
\]
then replacing $I_{N,T}$ by $\widehat I_{N,T}$ changes the second-branch bound to
\begin{equation}\label{eq:total-numerical-error}
|z(\varepsilon)-\widehat z_{\rm num}(\varepsilon)|
\le 2C'T^{-1/2}+E_{\rm quad}(\varepsilon).
\end{equation}
The constants in Lemmas~\ref{lem:series} and~\ref{lem:MBtail} arise from
asymptotic Gamma bounds and are not directly available to an implementation.
They prove convergence, but a machine-verifiable use of
\eqref{eq:total-numerical-error} would require explicit numerical majorants
and validated quadrature.

An implementable a posteriori control follows instead from the conditioning
of the active branch. Define the normalized stationary residual
\[
 \mathcal R_\varepsilon(w):=w-1+\varepsilon w^{p-1}.
\]
If \(w\in[\rho_{p,\lambda}/y,1]\), then
Proposition~\ref{prop:active-conditioning} below and the mean-value theorem give
\begin{equation}\label{eq:normalized-residual-bound}
 |w-z(\varepsilon)|
 \le \frac{|\mathcal R_\varepsilon(w)|}{1-p/2},
 \qquad
 |yw-x_+(y)|
 \le \frac{y|\mathcal R_\varepsilon(w)|}{1-p/2}.
\end{equation}
Thus the right-hand side supplies a computable exact-arithmetic stopping test
for either the series or the MB candidate. A formally certified
floating-point result additionally requires directed rounding (or interval
arithmetic) for the interval-membership and residual evaluations.

For the truncated series, the coefficients are precomputed and Horner's rule
evaluates \(1-\sum_{n=1}^N a_n\varepsilon^n\) in \(O(N)\) operations and
\(O(1)\) auxiliary storage. Algorithm~\ref{alg:hybrid-two} combines this
evaluation with the residual test.

\begin{algorithm}[Residual-controlled positive-branch evaluation]
\label{alg:hybrid-two}
Given \(p\in(0,1)\), \(\lambda>0\), \(y\ge\tau_{p,\lambda}\), a normalized
tolerance \(\eta>0\), a switch \(0<\rho<R<u_\star\), and
\(\varepsilon(y)\le R\), proceed as follows.
\begin{enumerate}
\item Set \(\varepsilon=\lambda p y^{p-2}\), choose an initial order \(N\),
and precompute or update the coefficients \(a_n\). For the MB branch also
choose \(\sigma_N\in(N,N+1)\), \(T>0\), and a quadrature rule.
\item If \(\varepsilon\le\rho\), compute the series candidate
\(\widehat z=1-\sum_{n=1}^Na_n\varepsilon^n\). Otherwise compute the
residue-corrected candidate from \eqref{eq:trunc-map-shifted}.
\item If \(|\operatorname{Im}\widehat z|\) exceeds the prescribed arithmetic
tolerance, reject the candidate. Otherwise set \(w=\operatorname{Re}\widehat z\),
check that \(w\in[\rho_{p,\lambda}/y,1]\), and define
\[
 \mathcal E(y):=
 \frac{y|\mathcal R_\varepsilon(w)|}{1-p/2}.
\]
If \(\mathcal E(y)\le y\eta\), return
\(\widehat x_+(y)=yw\). Otherwise increase \(N\), or on the MB
branch increase \(T\) and refine the quadrature, and repeat.
\item If the admissibility test repeatedly fails, use monotone bisection on
\([\rho_{p,\lambda},y]\), which is always available on the active branch.
\end{enumerate}
In exact arithmetic, \eqref{eq:normalized-residual-bound} makes
\(\mathcal E(y)\) a rigorous a posteriori bound. Interval evaluation turns
the same test into a floating-point certificate.
\end{algorithm}

\subsection{Conditioning on the Active Proximal Branch}
\label{subsec:conditioning}

The activation threshold of the proximity operator must be distinguished from
the saddle--node value of the stationary equation. This distinction also
clarifies the numerical conditioning of the root that is relevant to the
proximal map.

\begin{proposition}[Uniform conditioning on the active branch]
\label{prop:active-conditioning}
Let
\[
 \rho_{p,\lambda}=\bigl(2\lambda(1-p)\bigr)^{1/(2-p)},
 \qquad
 \tau_{p,\lambda}=\frac{2-p}{2(1-p)}\rho_{p,\lambda}.
\]
For every \(y\ge\tau_{p,\lambda}\), the larger stationary root satisfies
\(x_+(y)\ge\rho_{p,\lambda}\), and
\[
 1-\frac p2
 \le \partial_x\Phi_{p,\lambda,y}\bigl(x_+(y)\bigr)<1,
 \qquad
 1<\frac{d x_+(y)}{dy}\le\frac{1}{1-p/2}.
\]
Moreover,
\[
 0<\varepsilon(y)\le q_pu_\star,
 \qquad q_p:=p\,2^{1-p}<1.
\]
Thus, for fixed \(p\), both the inverse sensitivity and the series parameter
are uniformly bounded throughout the active proximal domain.
\end{proposition}

\begin{proof}
The saddle--node of the stationary equation occurs at
\[
 x_{\rm sn}=\bigl(\lambda p(1-p)\bigr)^{1/(2-p)},
 \qquad
 y_{\rm sn}=\frac{2-p}{1-p}x_{\rm sn}.
\]
It lies strictly below the activation threshold because
\[
 \frac{\tau_{p,\lambda}}{y_{\rm sn}}
 =\frac12\left(\frac2p\right)^{1/(2-p)}>1.
\]
The last inequality is equivalent to \(p\,2^{1-p}<1\); it also follows because
\(p\mapsto\log(p\,2^{1-p})\) is strictly increasing on \((0,1)\) and has limit
zero at \(p=1\).

At activation, the positive global minimizer is
\(x_+(\tau_{p,\lambda})=\rho_{p,\lambda}\). The larger root increases with
\(y\), so \(x_+(y)\ge\rho_{p,\lambda}\). Since \(p-2<0\),
\[
 \partial_x\Phi_{p,\lambda,y}\bigl(x_+(y)\bigr)
 =1-\lambda p(1-p)x_+(y)^{p-2}
 \ge 1-\lambda p(1-p)\rho_{p,\lambda}^{p-2}
 =1-\frac p2.
\]
Implicit differentiation gives the asserted bounds on \(d x_+/dy\).
Finally, \(\varepsilon(y)\) decreases with \(y\), and the calculation in
Theorem~\ref{thm:series-larger-root} gives
\(\varepsilon(\tau_{p,\lambda})=q_pu_\star\).
\end{proof}

In particular, \(\Phi_{p,\lambda,y}\) is strictly increasing on
\([\rho_{p,\lambda},y]\), with
\(\Phi_{p,\lambda,y}(\rho_{p,\lambda})\le0<\Phi_{p,\lambda,y}(y)\);
equality on the left occurs only at activation. For
\(y>\tau_{p,\lambda}\), the root can therefore be computed robustly by
bisection. The series may nevertheless converge more slowly when \(q_p\) is
close to one, especially as \(p\uparrow1\). The residue-corrected
representation can reduce finite-order error, but its contour cost grows in
the hardest cases; bisection remains a robust fallback. The difficulty is not
a singular active root.

\subsection{Certified Use in Proximal Gradient}
\label{subsec:pg-certified}

Let
\[
 F(x)=h(x)+g(x),\qquad g(x)=\lambda\|x\|_p^p,
\]
where \(h:\mathbb R^n\to\mathbb R\) has an \(L\)-Lipschitz gradient for
some \(L>0\), and let \(\partial\) denote the limiting subdifferential. For
\(\alpha\in(0,1/L)\), set
\[
 Q_\alpha(y;x):=h(x)+\langle\nabla h(x),y-x\rangle
 +\frac{1}{2\alpha}\|y-x\|^2+g(y).
\]
The following standard implication is recorded to specify the two
certificates required of the scalar evaluator.

\begin{theorem}[Certified inexact proximal gradient]
\label{thm:lp-inexact}
Assume that \(F\) is bounded below and that, for sequences
\(\delta_k,\eta_k\ge0\), the computed iterates satisfy
\begin{align}
 Q_\alpha(x_{k+1};x_k)&\le F(x_k)+\delta_k,
 \label{eq:model-certificate}\\
 r_{k+1}&\in \nabla h(x_k)+\frac1\alpha(x_{k+1}-x_k)
              +\partial g(x_{k+1}),
 \qquad \|r_{k+1}\|\le\eta_k.
 \label{eq:residual-certificate}
\end{align}
If \(\sum_{k=0}^{\infty}\delta_k<\infty\) and \(\eta_k\to0\), then, with
\(c_\alpha:=1/(2\alpha)-L/2>0\),
\begin{align*}
 F(x_{k+1})&\le F(x_k)-c_\alpha\|x_{k+1}-x_k\|^2+\delta_k,\\
 \sum_{k=0}^{\infty}\|x_{k+1}-x_k\|^2&<\infty,
 \qquad F(x_k)\ \text{converges},\\
 \operatorname{dist}\bigl(0,\partial F(x_{k+1})\bigr)
 &\le \eta_k+\left(L+\frac1\alpha\right)\|x_{k+1}-x_k\|
 \longrightarrow0.
\end{align*}
Consequently, every cluster point of \(\{x_k\}\) is a critical point of \(F\).
\end{theorem}

\begin{proof}
The descent lemma gives
\(F(y)\le Q_\alpha(y;x)-c_\alpha\|y-x\|^2\), so
\eqref{eq:model-certificate} yields the first inequality. For
\(A_k:=F(x_k)+\sum_{j=k}^{\infty}\delta_j\), it follows that
\(A_{k+1}\le A_k-c_\alpha\|x_{k+1}-x_k\|^2\); hence \(A_k\) and
\(F(x_k)\) converge and the squared steps are summable. Choosing
\(u_{k+1}\in\partial g(x_{k+1})\) in
\eqref{eq:residual-certificate} gives
\(w_{k+1}:=\nabla h(x_{k+1})+u_{k+1}
=r_{k+1}+\nabla h(x_{k+1})-\nabla h(x_k)
-\alpha^{-1}(x_{k+1}-x_k)\in\partial F(x_{k+1})\), and hence the
stated distance bound. If
\(x_{k_j}\to\bar x\), then \(x_{k_j+1}\to\bar x\) and
\(w_{k_j+1}\to0\); continuity of \(F\) and closedness of \(\partial F\)
give \(0\in\partial F(\bar x)\).
\end{proof}

The paper-specific point is that both certificates can be checked from the
scalar proximal calculations. Define
\[
 v_k:=x_k-\alpha\nabla h(x_k),\qquad
 \Psi_k(z):=\frac1{2\alpha}\|z-v_k\|^2+g(z).
\]
Completing the square gives
\[
 Q_\alpha(y;x_k)-F(x_k)=\Psi_k(y)-\Psi_k(x_k),
\]
so \eqref{eq:model-certificate} is the directly computable, separable test
\(\Psi_k(x_{k+1})\le\Psi_k(x_k)+\delta_k\). For a nonzero coordinate
\[
 x_{k+1,i}=\operatorname{sgn}(v_{k,i})\,\widehat s_i
\]
the corresponding residual is
\[
 r_{k+1,i}
 =\frac{\operatorname{sgn}(v_{k,i})}{\alpha}
   \Phi_{p,\alpha\lambda,|v_{k,i}|}(\widehat s_i).
\]
At a zero coordinate, \(\partial(\lambda|\cdot|^p)(0)=\mathbb R\), so the
coordinate residual can always be chosen as zero. The residual test alone
therefore cannot validate a zero output; this is the role of the model-value
test.

Let \(\mathcal I_k:=\{i:x_{k+1,i}\ne0\}\), and suppose that each
\(i\in\mathcal I_k\) is evaluated on the active branch, with
\(|v_{k,i}|\ge\tau_{p,\alpha\lambda}\) and
\(\widehat s_i\in[\rho_{p,\alpha\lambda},|v_{k,i}|]\). Setting
\(s_i=x_+(|v_{k,i}|)\), Proposition~\ref{prop:active-conditioning} and the
mean-value theorem give
\[
 \frac{1-p/2}{\alpha}|\widehat s_i-s_i|
 \le |r_{k+1,i}|
 \le \frac1\alpha|\widehat s_i-s_i|.
\]
Thus, an evaluator guarantee
\(|\widehat s_i-s_i|\le\mathcal E_i\) implies
\(|r_{k+1,i}|\le\mathcal E_i/\alpha\). Since the zero-coordinate residuals
may be chosen as above,
\[
 \|r_{k+1}\|
 \le \frac1\alpha
 \left(\sum_{i\in\mathcal I_k}\mathcal E_i^2\right)^{1/2}.
\]
Given prescribed sequences with \(\sum_k\delta_k<\infty\) and
\(\eta_k\to0\), the evaluator can therefore be refined until the separable
model-value and residual tests meet these tolerances;
Theorem~\ref{thm:lp-inexact} then applies. This formulation
avoids transferring limiting subgradients from an exact proximal point to a
nearby approximation, which is not justified by a distance bound alone, and
is consistent with standard descent-and-relative-error analyses of nonconvex
proximal algorithms
\citep{AttouchBolteSvaiter2013KL,BolteSabachTeboulle2014,LiPong2016}.

\subsection{Numerical Validation}
\label{sec:num-hybrid}

We report two complementary tests: a scalar parameter sweep and an
inexact proximal-gradient experiment. Reference scalar roots are computed
independently by monotone bisection at 70 and 100 decimal digits. The reported
errors are empirical floating-point errors, not validated interval
certificates.

For the scalar study, write \(z=x_+(y)/y\) and
\(e_z=|\widehat z-z_{\rm ref}|\). We consider the 135 combinations
\[
\begin{aligned}
p&\in\{0.1,\tfrac13,\tfrac12,\tfrac23,0.8,0.9,0.95,0.99,0.999\},\\
\lambda&\in\{10^{-6},0.8,10^6\},\\
y/\tau_{p,\lambda}-1&\in\{10^{-10},10^{-6},10^{-2},0.5,9\}.
\end{aligned}
\]
Besides fixed truncation orders, an adaptive series doubles \(N\), starting
from 16 and stopping at 16384, until its conditioning-based residual indicator
is at most \(10^{-11}\). Table~\ref{tab:numerical-validation}(a) shows that the
adaptive series meets this target in 129 cases. The six capped cases all have
\(p=0.999\) and relative activation gap \(10^{-10}\) or \(10^{-6}\), confirming
that slow series convergence as \(p\uparrow1\), rather than root
ill-conditioning, is the difficult regime. Bisection and safeguarded Newton
meet the reporting target in all 135 cases.

For a representative residue-corrected calculation with
\(p=\tfrac23\), \(\lambda=0.8\), and \(y=1.26\), the correction reduces
\(|\widehat x-x_{\rm ref}|\) from \(1.64\times10^{-5}\) to
\(9.80\times10^{-7}\) at \(N=30\), and from \(1.96\times10^{-7}\) to
\(2.25\times10^{-8}\) at \(N=50\). This illustrates finite-order
improvement, not uniform computational superiority. The supplementary
\(p=0.99\) tests require increasing contour effort; bisection or safeguarded
Newton is preferable in the hardest cases. Detailed contour-height and
quadrature tests are provided in the supplementary material.

For the algorithmic test, we minimize
\(F(x)=\tfrac12\|Ax-b\|^2+\lambda\|x\|_p^p\) on two Gaussian
\(40\times70\) instances with six-sparse ground truths and noise level
\(0.01\). For \(p\in\{0.5,0.8\}\), we use \(x_0=0\), \(\lambda=0.04\),
\(\alpha=0.9/\|A\|_2^2\), and the certificates of
Theorem~\ref{thm:lp-inexact} with
\(\delta_k=10^{-12}/(k+1)^2\) and
\(\eta_k=10^{-4}/(k+1)^{1.25}\). The stationarity measure is
\[
 S(x):=\left\|\bigl((A^{\mathsf T}(Ax-b))_i
 +\lambda p\,\operatorname{sgn}(x_i)|x_i|^{p-1}\bigr)_{i:x_i\ne0}
 \right\|_2.
\]

All eight series and Newton runs terminate within 177--198 iterations. Across
their 1506 accepted steps, the model-value certificate, residual certificate,
and predicted descent inequality all hold; the largest observed ratio
\(\|r_{k+1}\|/\eta_k\) is \(2.31\times10^{-2}\). No Newton fallback is used
by the series evaluator. Thus, the computable tests in
Section~\ref{subsec:pg-certified} can be enforced in a nonconvex
proximal-gradient calculation. Full traces and reproducibility details are
given in the supplementary material. MATLAB routines for the series and
residue-corrected formulas are available in the
\href{https://github.com/lshen03/Lp-Proximity-Operator}{project repository}.

\begin{table}[h!]
\centering
\caption{Numerical validation. In panel (a), ``accurate'' means
\(e_z\le10^{-11}\). Panel (b) reports the series-based proximal-gradient
runs; safeguarded Newton gives the same \(K\), \(F(x_K)\), and \(S(x_K)\)
to the displayed digits.}
\label{tab:numerical-validation}
\scriptsize
\setlength{\tabcolsep}{2.6pt}
\begin{minipage}[t]{0.53\linewidth}
\centering
\textbf{(a) Scalar accuracy}\\[2pt]
\begin{tabular}{lcc}
\toprule
Evaluator & Accurate/135 & Largest \(e_z\)\\
\midrule
Series, \(N=10\)  & 30  & \(2.13\times10^{-2}\)\\
Series, \(N=50\)  & 72  & \(4.63\times10^{-3}\)\\
Series, \(N=110\) & 90  & \(2.05\times10^{-3}\)\\
Adaptive series    & 129 & \(1.00\times10^{-7}\)\\
Bisection          & 135 & \(8.82\times10^{-12}\)\\
Safeguarded Newton & 135 & \(6.79\times10^{-12}\)\\
\bottomrule
\end{tabular}
\end{minipage}\hfill
\begin{minipage}[t]{0.44\linewidth}
\centering
\textbf{(b) Proximal-gradient runs}\\[2pt]
\begin{tabular}{ccrcc}
\toprule
Seed & \(p\) & \(K\) & \(F(x_K)\) & \(S(x_K)\)\\
\midrule
1 & 0.5 & 177 & 0.239049 & \(1.62\times10^{-7}\)\\
1 & 0.8 & 190 & 0.248428 & \(1.76\times10^{-7}\)\\
2 & 0.5 & 198 & 0.263414 & \(1.06\times10^{-7}\)\\
2 & 0.8 & 188 & 0.200700 & \(1.09\times10^{-7}\)\\
\bottomrule
\end{tabular}
\end{minipage}
\end{table}
\FloatBarrier

\section{Conclusions}
\label{sec:conclusion}

We derived an exact Lagrange--B\"urmann series for the larger stationary root
of the scalar nonconvex \(\ell_p\) proximal problem and proved uniform
convergence on the full active proximal domain. For rational \(p\), an explicit
residue-class formula reduces the series to a finite sum of generalized
hypergeometric functions, including compact formulas for
\(p\in\{\tfrac13,\tfrac12,\tfrac23\}\).
A Mellin--Barnes representation gives an exact residue correction and
separates analytic truncation error from quadrature error; an active-branch
residual supplies a computable a posteriori stopping test.

The activation threshold is strictly above the saddle--node threshold. As a
result, the active proximal root is uniformly well-conditioned; the main
near-activation issue is the series convergence factor, particularly when
\(p\) approaches one. Finally, a model-descent test and a computable
stationarity residual yield a rigorous inexact proximal-gradient result:
objective values converge, successive steps vanish, and every cluster point
is critical under the stated error conditions. The scalar parameter sweep
confirms the predicted slowdown as \(p\) approaches one, while the
sparse-recovery tests show that the two acceptance certificates can be
enforced in practice.

\section*{Declarations}

\noindent\textbf{Funding.}
The work of Lixin Shen was supported in part by the National
Science Foundation under Grant No.~DMS-2208385.

\medskip
\noindent\textbf{Competing interests.}
The authors have no relevant financial or non-financial interests
to disclose.

\medskip
\noindent\textbf{Data availability.}
No external datasets were used. The numerical experiments use synthetic data
generated from the random seeds specified in the supplementary material.

\medskip
\noindent\textbf{Code availability.}
The MATLAB implementation is available at
\url{https://github.com/lshen03/Lp-Proximity-Operator}.

\bibliographystyle{plainnat}
\bibliography{reference_fixed}

\end{document}